\documentclass[12pt]{amsart}

\usepackage{amsmath}
\usepackage{amsthm}
\usepackage{amsfonts}
\usepackage{xcolor}
\usepackage[english]{babel}
\usepackage[margin=1.3in]{geometry}

\usepackage[latin9]{inputenc}
\usepackage{amsmath}
\usepackage{amsfonts}
\usepackage{amssymb}
\usepackage{amsthm}

\newtheorem{theo}{Theorem}[section]
\newtheorem{prop}[theo]{Proposition}

\newtheorem{lemm}[theo]{Lemma}

\theoremstyle{definition}

\theoremstyle{remark}
\newtheorem{rema}[theo]{Remark}
\newcommand{\Op}{\operatorname{Op}}

\newcommand{\nwc}{\newcommand}
\nwc{\eps}{\epsilon}
\nwc{\vareps}{\varepsilon}
\nwc{\Oph}{\operatorname{Op}_\hbar}
\nwc{\la}{\langle}
\nwc{\ra}{\rangle}

\nwc{\mf}{\mathbf} 
\nwc{\blds}{\boldsymbol} 
\nwc{\ml}{\mathcal} 

\nwc{\defeq}{\stackrel{\rm{def}}{=}}

\nwc{\cE}{\ml{E}}
\nwc{\cN}{\ml{N}}
\nwc{\cO}{\ml{O}}
\nwc{\cP}{\ml{P}}
\nwc{\cU}{\ml{U}}
\nwc{\cV}{\ml{V}}
\nwc{\cW}{\ml{W}}
\nwc{\tU}{\widetilde{U}}
\nwc{\IN}{\mathbb{N}}
\nwc{\IR}{\mathbb{R}}
\nwc{\IZ}{\mathbb{Z}}
\nwc{\IC}{\mathbb{C}}
\nwc{\IT}{\mathbb{T}}
\nwc{\tP}{\widetilde{P}}
\nwc{\tPi}{\widetilde{\Pi}}
\nwc{\tV}{\widetilde{V}}
\nwc{\supp}{\operatorname{supp}}
\nwc{\rest}{\restriction}

\newcommand{\R}{{\mathbb R}}

\newcommand{\half}{{\frac{1}{2}}}

\renewcommand{\phi}{\varphi}

\title[$L^p$ norms, nodal sets, and quantum ergodicity]
{$L^p$ norms, nodal sets, and quantum ergodicity}

\author{Hamid Hezari }
\address{Department of Mathematics, UC Irvine, Irvine, CA 92617, USA} \email{hezari@math.uci.edu}

\author[Gabriel Rivi\`ere]{Gabriel Rivi\`ere}
\address{Laboratoire Paul Painlev\'e (U.M.R. CNRS 8524), U.F.R. de Math\'ematiques, Universit\'e Lille 1, 59655 Villeneuve d'Ascq Cedex, France}
\email{gabriel.riviere@math.univ-lille1.fr}

\date{\today}

\begin{document}

\begin{abstract} For small range of $p>2$, we improve the $L^p$ bounds of eigenfunctions of the Laplacian on negatively curved manifolds. Our improvement is by a power of logarithm for a full density sequence of eigenfunctions. We also derive improvements on the size of the nodal sets. 
Our proof is based on a quantum ergodicity property of independent interest, which holds for families of symbols supported in balls whose radius shrinks at a logarithmic rate.



\end{abstract}

\maketitle
\section{Introduction}
Let $(M,g)$ be a smooth boundaryless compact connected Riemannian manifold of dimension $d$, and let $\Delta_g$ be the Laplace-Beltrami operator. The eigenfunctions $\psi_\lambda$ of $\Delta_g$ are nonzero solutions to
$$ -\Delta_g \psi_\lambda = \lambda \psi_\lambda, \quad \lambda \geq 0. $$
 A classical problem is to study the asymptotic properties of $\psi_{\lambda}$ as $\lambda\rightarrow+\infty$, and their relations to the geometry of the 
 manifold $(M,g)$. For instance, one can study the size of their $L^p$ norms when $\psi_\lambda$ is $L^2$ normalized, or study the geometry of their nodal sets
$$Z_{\psi_\lambda}=\{ x \in M;  \; \psi_\lambda(x)=0\}. $$
Another interesting problem is to study the asymptotic of the positive measures $|\psi_{\lambda}|^2dv_g$, where $v_g$ is the normalized volume measure on 
$M$. The purpose of this article is to establish a new asymptotic property of these measures on negatively curved manifolds, and to deduce from this some 
results on the $L^p$ norms and the size of nodal sets of eigenfunctions. This new property is that quantum ergodicity holds on small scale balls of radius 
$(\log \lambda)^{-K}$ with $0<K<\frac{1}{2d}$. 
One of the main observations of the present article is that any improvement on the radius of such balls would give improvements on $L^p$ estimates, and on the size of nodal sets.
 
Before we state our results we recall that since $M$ is compact, there exist isolated eigenvalues with finite multiplicities
$$0=\lambda_1<\lambda_2 \leq \lambda_3 \dots \to \infty,$$
and an orthonormal basis $(\psi_j)_{j\in\mathbb{N}}$ of $L^2(M)$ satisfying
$$ -\Delta_g \psi_j = \lambda_j \psi_{j}.$$

\subsection{$L^p$ norms}

Our first result gives upper bounds on the $L^p$ norms of eigenfunctions on negatively curved manifolds. More precisely, we will show that $L^p$ bounds of 
eigenfunctions can be improved along a full density subsequence for small values of $p$.

Our main result on $L^p$ norms of eigenfunctions is the following:
\begin{theo}\label{theorem2} Let $(M,g)$ be a compact negatively curved manifold. Then, for every $K \in (0, \frac{1}{d})$, there exists $C_{K,p}>0$ such that, for every ONB $\{\psi_{j}\}$ of $\Delta_g$ eigenfunctions, there exists a full density subset $S$ of $ \mathbb{N}$ such that 
\begin{equation} \label{LpEstimates} j \in S, \quad p \in (2, \infty]: \quad  \| \psi_{j} \|_{L^p(M)}  \leq C_{K,p}\Big (\frac{\lambda_j}{{(\log \lambda_j})^{K}} \Big)^{\gamma(p)},\end{equation}  where $$ \gamma(p) = \max \left \{\frac{d-1}{4}\left (\frac{1}{2} - \frac{1}{p}\right ), \frac{d-1}{4} -\frac{d}{2p} \right \},$$is the Sogge exponent.  
\end{theo}
Recall that a subset $S$ of $\mathbb{N}$ is said to have full density if
$$\lim _{N \to \infty} \frac{\text{card}\,(S\cap [1,N] )}{N}=1.$$
On a general smooth compact Riemannian manifold (without boundary), the above estimates are valid for any eigenfunction but without the logarithmic factor~\cite{So88}. Moreover, these estimates are sharp without any further geometric assumptions. Eigenfunctions $L^p$-estimates and their applications have been extensively studied in the literature. For the sake of brevity we only list some relevant articles for interested readers: \cite{So88, So93, KTZ07, So11, SoZe11, STZ11, BlSo13, HT13}.

Our result is in fact of particular interest in the range $2<p\leq\frac{2(d+1)}{d-1}$, where to our knowledge no logarithmic improvements are known. 
However in the range $\frac{2(d+1)}{d-1}< p < \infty$, our result is weaker than the recent results of Hassell and Tacy~\cite{HT13} where the upper 
bound $\frac{\lambda_j^{\gamma (p)}}{(\log \lambda_j)^{\half}}$ is proved for non-positively curved manifolds and for all $\lambda_j \in \text{Spec} 
\Delta_g$. We also note that, for this range of $p$, Sogge's $L^p$ bounds can be improved by a polynomial factor in the case of the rational torus~\cite{Bo93, BoDe14}.

The only former results concerning the range $2<p\leq \frac{2(d+1)}{d-1}$ which we are aware of, are the results of Zygmund \cite{Zy74}, Bourgain~\cite{Bo93, Bo13}, Sogge~\cite{So11}, 
Sogge and Zelditch~\cite{SZ12b, SoZe11} and Blair and Sogge~\cite{BlSo13}. Let us recall briefly the results of Blair, Sogge and Zelditch which work for geometric frameworks close to ours. In~\cite{So11, BlSo13}, the 
upper bound $o(\lambda_j^{\gamma(p)})$ is proved for the $L^p$ norm of a full density subsequence of eigenfunctions when $2<p< \frac{2(d+1)}{d-1}$ and the 
geodesic flow is ergodic for the Liouville measure. In~\cite{SZ12b}, it is shown that this result remains true for the $L^4$ norm on compact \emph{surfaces} 
satisfying a weaker dynamical assumption; that the set of closed trajectories has zero Liouville measure on $S^*M$. In~\cite{SoZe11, BlSo13}, the upper bound $o(\lambda_j^{\gamma(p)})$ 
is proved for the $L^p$ norm of the whole sequence of eigenfunctions provided $2<p< \frac{2(d+1)}{d-1}$ and $(M,g)$ is nonpositively curved. Thus, the main improvements compared with 
these references are that our method allows to treat the case of the critical exponent $p=\frac{2(d+1)}{d-1}$ and that it gives a logarithmic improvement on the size of the upper bound for the whole range $2<p\leq \frac{2(d+1)}{d-1}$. On the other hand, compared with~\cite{SoZe11, BlSo13}, our upper bounds are only valid along density $1$ subsequences of eigenfunctions.

Finally, we recall that the results of Zygmund and Bourgain concern the particular case of the rational torus 
$\IT^d=\IR^d/\IZ^d$. Their results shows that, for $2<p\leq\frac{2d}{(d-1)}$, the upper bound is in fact of order $\ml{O}(\lambda^{\eps})$ for every fixed 
positive $\eps$. In the case $d=2$, one can even take $\eps=0$~\cite{Zy74} and the upper bound remains true up to $p=+\infty$ provided we keep $\eps>0$~\cite{Bo93}. In the case $d=3$, 
the upper bound remains true up to $p=\frac{2(d+1)}{d-1}$~\cite{Bo13}. 

\begin{rema}
Our method also gives a logarithmic improvement on $L^\infty$ bounds (See paragraph~\ref{sss:Linfinity}). Yet, our estimates are weaker than B\'erard's upper bound   $\frac{\lambda_j^{(d-1)/4}}{(\log \lambda_j)^{\half}}$, \cite{B77}, but nevertheless it is interesting that $L^\infty$ bounds can be improved only using quantum ergodicity on small balls. We also note that we can obtain better $L^p$ estimates in the range $p>\frac{2(d+1)}{d-1}$, by interpolating B\'erards $L^\infty$ estimate and our $L^p$ estimate for $p=\frac{2(d+1)}{d-1}$, but we omit this because it will not give us a better estimate than those of \cite{HT13}.
\end{rema}

\textbf{Comments on other geometric settings}. Recall that quantum ergodicity was first proved for ergodic Hamiltonian systems in~\cite{Sh74, Ze87, CdV85, HeMaRo87}. As was already mentionned, our proof relies on a quantum ergodicity property that holds for symbols depending only on the $x$ variable and carried in balls of shrinking radius. To our knowledge, this particular form of quantum ergodicity has not been studied before except in the recent preprint by Han~\cite{Han14} (which was proved independently) -- see also~\cite{Yo13} 
for related results in the arithmetic setting. To be more precise, our quantum ergodicity theorem holds for symbols supported in balls of shrinking radius $\eps\sim|\log\lambda|^{-K}$, 
with $0<K<1/(2d)$.  Here, our logarithmic improvement in the size of the support of the symbols is the result of using the semiclassical approximation up to the Ehrenfest time as in~\cite{Ze94, Sch06} -- see also~\cite{AnRi12} for related results.

It is natural to ask if this approach can be used in other situations where one has equidistribution for observables depending only on the $x$ variable. 
For instance, it is known that such a property holds in the case of the torus~\cite{MaRu12, Ri13}. In~\cite{HeRi15}, we show how to
obtain a small scale equidistribution property in the case of the rational torus. However, the $L^p$  bounds we would get from this method would not be better 
than the ones of ~\cite{Bo13, BoDe14, Zy74}. Another interesting case is the one of Hecke eigenfunctions on the modular surface $M=PSL_2(\IZ)\backslash\mathbb{H}^2$ which 
is non-compact. On one hand, the best local $L^p$ bounds for finite $p$ can be obtained (as far as we know) by interpolating Sogge's 
$L^{\frac{2(d+1)}{d-1}}$-upper bounds with the $L^{\infty}$-upper bound from~\cite{IwSa95}. On the other hand, Luo and Sarnak also proved a very precise rate for the quantum 
variance of Hecke eigenfunctions in their quantum ergodicity theorem~\cite{LuSa95}. In particular, a direct corollary of their result is that quantum ergodicity still holds for 
observables carried in shrinking balls of radius $\lambda^{-\nu}$ with $\nu>0$ small enough -- see also~\cite{Yo13} for related results under the Lindel\"of hypothesis . If one implements this 
observation in the argument of section~\ref{s:proof} below, one would obtain a $L^{\frac{2(d+1)}{d-1}}$-upper bound\footnote{The $L^p$ upper bounds are 
valid for compact subsets of $PSL_2(\IZ)\backslash\mathbb{H}^2$.} improved by a small polynomial factor compared with Sogge's result (for a full density subsequence of eigenfunctions). Then, 
interpolating this result with the $L^2$ norm and the $L^{\infty}$ upper bounds from~\cite{IwSa95} or~\cite{Ju13}, one would get upper bounds for every $p\geq 2$ which would be valid for a 
full density subsequence of eigenfunctions and which would be slightly better than the ones mentionned above. However, it is plausible that a direct application of arithmetic tools would
provide better $L^p$ upper bounds for small range of $p$ but we are not aware of such results.

\subsection{Nodal sets}

As another application of our small scale quantum ergodicity properties, we will obtain lower bounds on the $(d-1)$-dimensional Hausdorff measure ${\mathcal H} ^{d-1} (Z_{\psi_\lambda})$ of the nodal sets.
\begin{theo}\label{theorem} Let $(M,g)$ be a smooth boundaryless compact connected Riemannian negatively curved manifold of dimension $d$, and $\{\psi_{j}\}$ an ONB of eigenfunctions. Then for every $\delta>0$ there exists a full density subset $S \subset \mathbb{N}$ such that for every open set $U \subset M$ we have
$$\forall j \in S: \qquad 
\mathcal{H} ^{d-1} (Z_{\psi_{j}} \cap U) \geq c\;
(\log \lambda_j)^{\frac{d-1}{4d}-\delta}\lambda_j^{\frac{3-d}{4}} .
$$
Here $c$ is positive and only depends on $\delta$ and the inner radius of $U$. In particular, in the 3-dimensional case we get $$ \forall j \in S: \qquad \mathcal{H} ^{2} (Z_{\psi_{j}} \cap U) \geq c\;
(\log \lambda_j)^{\frac{1}{6}-\delta} .$$
\end{theo}


\textbf{Background on nodal sets}. For any smooth boundaryless compact connected Riemannian manifold of dimension $d$, Yau's conjecture states that there exist constants $c>0$ and $C>0$ independent of $\lambda$ such that
$$c \sqrt{\lambda} \leq \mathcal H ^{d-1} (Z_{\psi_\lambda}) \leq C \sqrt{\lambda}.$$
The conjecture was proved by Donnelly and Fefferman \cite{DF} in the real analytic case. In dimension $2$ and in the $C^\infty$ case, the best bounds are
$$c\sqrt{\lambda} \leq \mathcal H ^{1} (Z_{\psi_\lambda}) \leq C \lambda ^{3/4}.$$
The lower bound was proved by Br\"uning \cite{B} and Yau. The upper bound for $d=2$ was proved by Donnelly-Fefferman \cite{DF2} and Dong \cite{D}. For $d \geq 3$ the existing estimates are very far from the conjecture. The best lower bound is
$$c \lambda^{(3-d)/4} \leq \mathcal H ^{d-1} (Z_{\psi_\lambda}), $$ which was first proved by \cite{CM}, and later by \cite{HS12} and \cite{SZ12}. In \cite{HS12} the following slightly better estimate was proved: 
\begin{equation}\label{e:HeSo} c \sqrt {\lambda} ||\psi_\lambda||_{L^1}^2 \leq \mathcal H ^{d-1} (Z_{\psi_\lambda}),\end{equation} 
for any $L^2$-normalized eigenfunction $\psi_\lambda$. As was already mentionned, Blair and Sogge obtained improvements on the $L^p$ norms of eigenfunctions for nonpositively curved manifolds and quantum ergodic sequences of eigenfunctions. In~\cite{BlSo13}, they recall how one can deduce from these results improved lower bounds for the $L^1$ norm of eigenfunctions and (thanks to~\eqref{e:HeSo}) improvements on the results of~\cite{CM, HS12, SZ12}. It is conjectured in \cite{SZ11} that for negatively curved manifolds the $L^1$ norm is bounded below by a uniform constant. In \cite{Ze13}, this conjecture is proved under the assumption of the so called ``$L^\infty$ quantum ergodicity'', which is stronger than the standard ``$C^0$ quantum ergodicity'' and, to our knowledge, it has not been proved or disproved in any ergodic cases. In the same article it is in fact conjectured that $L^\infty$ quantum ergodicity holds on negatively curved manifolds. 
Our logarithmic improvement on the lower bounds of the nodal sets in the case of negatively curved manifolds is far from the above conjecture but
it uses again quantum ergodicity in a stronger sense than usual. Precisely, it uses the fact that the eigenfunctions are still equidistributed on balls
of shrinking radius $\eps\sim|\log\lambda|^{-K}$, with $0<K<1/(2d)$. Finally, we emphasize that, compared with the above results, our result is valid for the size of the nodal sets on any open set $U$ inside $M$.

\subsection{Outline of the proof} 
The idea of the proof of Theorem~\ref{theorem} is as follows. We cover $M$ with geodesic balls $B$ of radius $\eps=(\log \lambda)^{-K}$. We then fix such a ball $B$ and use our refined version of quantum ergodicity to obtain a full density subsequence of eigenfunctions which take $L^2$ mass on $B$ comparable to vol$(B)$. Then, we apply the method of \cite{CM} and obtain a lower bound for the size of the nodal set in this shrinking ball. This latter property requires to control the $L^p$ norm of the eigenfunctions inside these shrinking ballls. For that purpose, we rescale our problem and make use of semiclassical $L^p$ estimates for quasimodes~\cite{So88, KTZ07, Zw12}. As a corollary of this intermediary step, we also get the proof of Theorem~\ref{theorem2}. Finally, we conclude the proof of Theorem~\ref{theorem} by adding up over all balls in the covering. We observe that, since the number of balls is $\eps^{-d}$ one has to be careful in choosing a full density subsequence that works uniformly for all balls in the covering. 


\begin{rema}
 After the initial posting of this article, Sogge in ~\cite{So15} was able to improve our arguments in paragraph~\ref{ss:Lp} in order to obtain localized $L^p$ estimates. These estimates establish a clear relation between $L^p$ bounds of eigenfunctions and estimating $L^2$ mass of eigenfunctions in balls with shrinking radius. We refer to his note for a precise statement.
\end{rema}

The article is organized as follows. In Section~\ref{s:QE}, we prove our refined version of quantum ergodicity on negatively curved manifolds, and then, in Section~\ref{s:proof}, we apply these results to the study of $L^p$ norms and nodal sets. In Appendix~\ref{a:moments}, we give the proof of a dynamical result on the rate of convergence of Birkhoff averages with respect to the $L^p$ norm.




\begin{rema}
Since we are going to use semiclassical techniques we prefer to use the semiclassical parameter $\hbar$  and we will use this for the rest of this paper. For readers who are not familiar with the semiclassical parameter $\hbar$, one can think of $\hbar$ as $1/\sqrt{\lambda}$. Hence the high energy regime $\lambda \to \infty$ is equivalent to the semiclassical limit $\hbar \to 0$. 
\end{rema}

\section{Quantum ergodicity on small scale balls}
\label{s:QE}

In all of this section, $M$ is a smooth, connected, compact Riemannian manifold without boundary and \textbf{with negative sectional curvature}.

Let $\alpha>0$ be some fixed positive number. For every $0<\hbar\leq 1$, we consider an orthonormal basis (made of eigenfunctions of $-\hbar^2\Delta_g$) $(\psi^{j}_{\hbar})_{j=1,\ldots, N(\hbar)}$ of the subspace 
$$\ml{H}_{\hbar}:=\mathbf{1}_{[1-\alpha\hbar,1+\alpha\hbar]}(-\hbar^2\Delta_g)L^2(M).$$
According to the Weyl's law for negatively curved manifolds (see \cite{DuGu75, B77}), one has $N(\hbar)\sim  \alpha A_M\hbar^{1-d}$ for some constant $A_M$ depending only on $(M,g)$. For each $1\leq j\leq N(\hbar)$, we denote $E_j(\hbar) \in [1-\alpha\hbar,1+\alpha\hbar]$ to be the eigenvalue corresponding to $\psi_j^\hbar$:
$$-\hbar^2\Delta\psi_{\hbar}^j=E_j(\hbar)\psi_{\hbar}^j.$$

Let $0\leq \chi\leq $ be a smooth cutoff function which is equal to $1$ on $[-1,1]$ and to $0$ outside $[-2,2]$. For a given $x_0$ in $M$ and a $0<\eps<\frac{\text{inj}(M,g)}{10}$, we define
$$\chi_{x_0,\eps}(x)= \chi\left(\frac{\|\exp_{x_0}^{-1}(x)\|_{x_0}}{\epsilon}\right),$$ 
where $\exp_{x_0}$ is the exponential map associated to the metric $g$, and $\|v\|_{x_0}^2=g_{x_0}(v,v)$. By construction, this function is compactly supported in $B(x_0,2\eps)$. The purpose of this section is to describe the rate of convergence of the following quantity:
 $$V_{\hbar,p}(x_0,\eps):=\frac{1}{N(\hbar)}\sum_{j=1}^{N(\hbar)}\left|\int_{M}\chi_{x_0,\eps}|\psi_{\hbar}^j|^2dv_g-\int_{M}\chi_{x_0,\eps}dv_g\right|^{2p},$$
where $p$ is a fixed positive integer and $v_g$ is the normalized volume measure on $M$. Observe that there exists uniform (in $x_0$ and in $\eps$) constants $0<\tilde{c}_1<\tilde{c}_2$ such that
\begin{equation}\label{e:comparison}\tilde{c}_1\eps^{d}\leq\int_{M}\chi_{x_0,\eps}dv_g\leq \tilde{c}_2\eps^d.\end{equation}
It follows from~(\ref{e:comparison}) that to obtain some nontrivial rate we need to show that $V_{\hbar,p}(x_0,\eps)$ is at least $o(\eps^{2dp})$. We will show that this is in fact possible for a particular choice of $\eps$ as a function of $\hbar$. 

Our main result on the rate of convergence of $V_{\hbar,p}(x,\eps)$ to zero is the following:
\begin{prop}\label{p:QE} Let $M$ be a smooth, connected, compact Riemannian manifold without boundary and with negative sectional curvature. Let $0<K<\frac{1}{2d}$, let $m>0$ and $\beta>0$.\\
Then, there exists $C>0$ and $\hbar_0>0$ such that the following holds, for every $0<\hbar\leq\hbar_0$:\\
Given an orthonormal basis $(\psi^{j}_{\hbar})_{j=1,\ldots, N(\hbar)}$ of the subspace $\ml{H}_{\hbar}$ made of eigenfunctions of $-\hbar^2\Delta_g$, one has, for every $x_0\in M$,
$$V_{\hbar,p}(x_0,m|\log\hbar|^{-K})\leq \frac{C}{|\log\hbar|^{p(1-2K\beta)}}.$$

\end{prop}

The assumption $0<K<1/(2d)$ implies that the upper bound is (at least for $\beta>0$ small enough) $o(\eps^{2dp})$ as expected. We point out that the main results of the recent preprint~\cite{Han14} give a small scale quantum ergodicity result (with $p=1$) that holds for more general classes of symbols.

\subsection{Classical ergodicity}

Before giving the proof of Proposition~\ref{p:QE}, we need some precise information on dynamics of the geodesic flow of negatively curved manifolds. Let $0<\beta<1$, and denote by $\ml{C}^{\beta}(S^*M)$ the space of $\beta$-H\"older functions on $S^*M$. Then, one has the following property of exponential decay of correlations~\cite{Li04}:
\begin{theo}\label{t:liverani2} Let $M$ be a smooth, connected, compact Riemannian manifold without boundary and with negative sectional curvature. Let $0<\beta<1$.\\ 
Then, there exists some constants $C_{\beta}>0$ and $\sigma_{\beta}>0$ such that, for every $a$ and $b$ in $\ml{C}^{\beta}(S^*M)$, and for every $t$ in $\IR$,
$$\left|\int_{S^*M}(a\circ G^t) bdL-\int_{S^*M}adL\int_{S^*M}bdL\right|\leq C_{\beta}e^{-\sigma_{\beta}|t|}\|a\|_{\ml{C}^{\beta}}
\|b\|_{\ml{C}^{\beta}},$$
where $L$ is the normalized Liouville measure on $S^*M$, $G^t$ is the geodesic flow, and $\|.\|_{\ml{C}^{\beta}}$ is the H\"older norm. 
\end{theo}
This result on the rate of mixing has implications on the rate of classical ergodicity. More precisely, given $a$ in $\ml{C}^{\beta}(S^*M,\IR)$ and $T>0$, one can define, for every positive integer $p$, 
$$V_{p}(a,T):=\int_{S^*M}\left|\frac{1}{T}\int_0^Ta\circ G^tdt-\int_{S^*M}adL\right|^{2p}dL.$$
Using Theorem~\ref{t:liverani2}, one can easily show that $V_{1}(a,T)\leq \frac{C_{\beta,1}\|a\|_{\ml{C}^{\beta}}^{2}}{T}.$ Moreover, we can also control
the moments of order $2p$ but obtaining such estimates requires more work:
\begin{prop}\label{p:moments} Let $M$ be a smooth, connected, compact Riemannian manifold without boundary and with negative sectional curvature. Let $0<\beta<1$, and 
let $p$ be a positive integer. Then, there exists some constant $C_{\beta,p}>0$ such that, for all symbols $a$ in $\ml{C}^{\beta}(S^*M)$, and for every $T>0$,
 $$V_{p}(a,T):=\int_{S^*M}\left|\frac{1}{T}\int_0^Ta\circ G^tdt-\int_{S^*M}adL\right|^{2p}dL\leq \frac{C_{\beta,p}\|a\|_{\ml{C}^{\beta}}^{2p}}{T^p}.$$
\end{prop}
A slightly weaker version of this inequality can be obtained as a corollary of the central limit theorem proved in~\cite{Ze94, MeTo12}. In both references, the proof was based on results due to Ratner~\cite{Ra73}, and the dependence of the constant in terms of H\"older norms was not very explicit. In Appendix~\ref{a:moments}, we will use a slightly different approach which does not require the use of Markov partitions and symbolic dynamics, and which gives the above explicit constants in terms of the H\"orlder norm of the symbol. More precisely, we will use the machinery developed (for instance) by Liverani in~\cite{Li04} in order to study the spectral properties of the transfer operator. We emphasize that, unlike the case $p=1$, the rate of convergence of higher moments cannot be directly deduced from Theorem~\ref{t:liverani2}, and that we really need to use more precise results from~\cite{Li04} in order to get Proposition~\ref{p:moments}.

In the case where we pick the function $\chi_{x,\eps}$ defined above, we get
\begin{equation}\label{e:liverani2}V_p(\chi_{x,\eps},T):=\int_{S^*M}\left|\frac{1}{T}\int_0^T\chi_{x,\eps}\circ G^tdt-\int_{M}\chi_{x,\eps}dv_g\right|^{2p}dL=\ml{O}\left(\frac{1}{\eps^{2p\beta}T^p}\right),\end{equation}
where the constant in the remainder can be chosen uniformly in terms of $x$, $\eps$ and $T$. We observe that this is valid for any $\beta>0$ (of course the constant depends on $\beta$).

\subsection{Proof of Proposition~\ref{p:QE}}
\label{ss:proof-QE}
We will now implement these dynamical properties in the ``standard'' proof of quantum ergodicity~\cite{Sh74, Ze87, CdV85, HeMaRo87, Ze96, Zw12}. As in~\cite{Ze94, Sch06}, we will make use of the semiclassical approximation up to times of order $|\log\hbar|$.

Let $x$ be a point in $M$, let $0<K<\frac{1}{2d}$ and let $m>0$. We define 
$$\eps=\eps_{\hbar}:=m|\log\hbar|^{-K}.$$
We set $\overline{\chi}_{x,\eps}:=\chi_{x,\eps}-\int_{M}\chi_{x,\eps}dv_g$. In particular, for every $0<\nu<\frac{1}{2}$, $\overline{\chi}_{x,\eps}$ belongs to the ``nice'' class of symbols $S^{0,0}_{\nu}(T^*M)$ as defined in~\cite{Zw12} for instance. This class of symbols is suitable to pseudodifferential calculus (composition laws, Weyl's law, Egorov's theorem, etc.). 
\begin{rema}\label{r:cutoff}
We introduce $0\leq\chi_1\leq 1$ a smooth cutoff function which is identically $1$ in a small neighborhood of $1$, and which vanishes outside a slightly bigger neighborhood. We observe that, for $\hbar>0$ small enough, one has $\chi_1(-\hbar^2\Delta_g)\psi_{\hbar}^j=\psi_{\hbar}^j$ for every $1\leq j\leq N(\hbar)$. We also recall that $\chi_1(-\hbar^2\Delta_g)$ is a $\hbar$-pseudodifferential operator in $\Psi^{-\infty, 0}(M)$ with principal symbol $\theta(x,\xi):=\chi_1((\|\xi\|_x^*)^2)$~\cite{Zw12} (chapter $14$). We have denoted by $\|.\|_x^*$ the induced metric on $T_x^*M$.
\end{rema}
We then get
$$V_{\hbar,p}(x,\eps)=\frac{1}{N(\hbar)}\sum_{j=1}^{N(\hbar)}\left|\left\la\psi_{\hbar}^j,\Oph(\theta\overline{\chi}_{x,\eps})\psi_{\hbar}^j\right\ra\right|^{2p}+\ml{O}(\hbar),$$
where $\Oph$ is a fixed quantization procedure~\cite{Zw12}. Without loss of generality, we can suppose that $\Oph(1)=\text{Id}_{L^2(M)}$ and that it is a positive quantization procedure~\cite{Ze87, CdV85, HeMaRo87, Ze94}, i.e. $\Oph(a)\geq 0$ if $a\geq 0$. We now let 
$$T(\hbar)=\kappa_0|\log\hbar|,$$ 
with $\kappa_0>0$. Since $\psi_{\hbar}^j$ are $\Delta_g$-eigenfunctions, we can replace $\left\la\psi_{\hbar}^j,\Oph(\theta\overline{\chi}_{x,\eps})\psi_{\hbar}^j\right\ra$ by
$$\frac{1}{T(\hbar)}\int_0^{T(\hbar)}\left\la\psi_{\hbar}^j,e^{-it\hbar\Delta_g/2}\Oph(\theta\overline{\chi}_{x,\eps})e^{it\hbar\Delta_g/2}\psi_{\hbar}^j\right\ra dt.$$
We now choose and fix $\kappa_0$ small enough (only depending on $(M,g)$) and apply the Egorov theorem up to $T(\hbar)$(see \cite{Zw12}, section $11.4$) to find that
$$V_{\hbar,p}(x,\eps)=\frac{1}{N(\hbar)}\sum_{j=1}^{N(\hbar)}\left|\left\la\psi_{\hbar}^j,\Oph\left(\theta\frac{1}{T}\int_0^T\overline{\chi}_{x,\eps}\circ G^tdt\right)\psi_{\hbar}^j\right\ra\right|^{2p}+\ml{O}(\hbar^{\nu_0}),$$
for some fixed $\nu_0>0$ (depending on $\kappa_0>0$). Moreover, still thanks to the Egorov theorem, the principal symbol $\frac{1}{T}\int_0^T\overline{\chi}_{x,\eps}\circ G^tdt$ belongs to $S^{0,0}_{\nu}(T^*M)$ for some $0<\nu<1/2$. Because $\Oph$ is a positive quantization procedure, the distribution 
$$\mu_{\hbar}^j:b\in\mathcal{C}^{\infty}_c(T^*M)\mapsto\left\la\psi_{\hbar}^j,\Oph\left(b\right)\psi_{\hbar}^j\right\ra$$
is positive on $T^*M$, and in fact it is a positive measure on $T^*M$. Then by applying the Jensen's inequality (recall that $p\geq 1$) we get
$$V_{\hbar,p}(x,\eps)=\frac{1}{N(\hbar)}\sum_{j=1}^{N(\hbar)}\left\la\psi_{\hbar}^j,\Oph\left(\left|\theta\frac{1}{T}\int_0^T\overline{\chi}_{x,\eps}\circ G^tdt\right|^{2p}\right)\psi_{\hbar}^j\right\ra+\ml{O}(\hbar^{\nu_0}).$$
We now apply the local Weyl law  (see for instance Proposition $1$ in~\cite{Sch06} for a proof of this fact in our context). We get
$$V_{\hbar,p}(x,\eps)\leq C_0\int_{T^*M}\left( \theta \frac{1}{T}\int_0^T\overline{\chi}_{x,\eps}\circ G^tdt\right)^{2p}dL+\ml{O}(\hbar^{\nu_0}),$$
where $C_0$ is independent of $\hbar$ and $x$ in $M$. The parameter $\nu_0>0$ can become smaller from line to line but it remains positive. We underline that the different constants do not depend on the choice of the ONB of $\ml{H}_{\hbar}$. 
\begin{rema}
We also emphasize that, up to this point, all the constants in the remainder are uniform for $x$ in $M$ (by construction of the function $\chi_{x,\eps}$). Precisely, they only depend on a finite number derivatives of $\chi$, the manifold $(M,g)$, and the choice of the quantization procedure.
\end{rema}
We now apply property~\eqref{e:liverani2},  and we get that\footnote{Recall that $T=T(\hbar)=\kappa_0|\log\hbar|$.}
$$V_{\hbar,p}(x,\eps)\leq\ml{O}\left(\frac{1}{\eps^{2\beta p}|\log\hbar|^p}\right)+\ml{O}(\hbar^{\nu_0}),$$
which implies Proposition~\ref{p:QE}.

\section{Proof of the main results}
\label{s:proof}

Again, even if we do not mention it at every step, in this section $M$ will be a smooth, connected, compact Riemannian manifold without boundary and \textbf{with negative sectional curvature}.

In this section we give the proofs of our main results. The strategy is as follows. We start by covering the manifold with balls of radius $\eps=|\log\hbar|^{-K}$, and we apply our quantum ergodicity result from the previous section to extract a density $1$ subsequence ``adapted'' to this cover of $M$. We then improve the $L^p$ bounds along this subsequence of eigenfunctions by rescaling our problem and using the semiclassical $L^p$ estimates for quasimodes~\cite{So88, KTZ07, Zw12}. After that, we use the approach of Colding-Minicozzi ~\cite{CM} to prove our lower bound for the size of the nodal set on any open subset $U$.

\subsection{Covering $M$ with small balls}

First, we cover $M$ with balls $(B(x_k, \eps))_{k=1,\ldots, R(\eps)}$  of radius $\eps =|\log \hbar|^{-K}$, for a fixed $0<K<\frac{1}{2d}$. In addition, we require that the covering is chosen in such a way that each point in $M$ is contained in $C_g$-many of the double balls $B(x_k, 2\eps)$. The number $C_g$ can be chosen to be only dependent on $(M,g)$, and hence independent of $0<\eps\leq 1$. See for example Lemma $2$ in~\cite{CM} for a proof of this fact. We have
$$M \subset \bigcup_{k=1}^{R(\eps)} B(x_k, \eps).$$ The doubling property shows that there exists constants $c_1>0$ and $c_2>0$ , only dependent on $(M,g)$, such that
\begin{equation}\label{e:number-eps-ball} c_1 \eps^{-d} \leq R(\eps) \leq c_2 \eps^{-d}.\end{equation} 

We note that in fact we have a family of coverings parametrized by $\eps$, and hence the centers $\{x_k\} $ also depend on $\eps$. Now let $1\leq k\leq R(\eps)$. As in section~\ref{s:QE}, we define 
$$\chi_{x_k,\eps}(x)= \chi\left(\frac{\|\exp_{{x}_k}^{-1}(x)\|_{{x}_k}}{\epsilon}\right),$$ 
where $\exp_{x_k}$ is the exponential map associated to the metric $g$. By construction, this function is compactly supported in $B({x}_k,2\eps)$.

\subsection{Extracting a full density subsequence}\label{ss:density1}
Following the proof of~\cite{Zw12}, we want to extract a subsequence of density $1$ of eigenfunctions for which quantum ergodicity holds for all symbols $\chi_{x_k, \eps}$ uniformly in $k$. This is a standard procedure, and we just have to pay attention to the fact that the number of balls in the cover is large.

Let $m=\frac{1}{2}$ or $m=50$. Let $1\leq k\leq R(\eps)$, $0<K<\frac{1}{2d}$, and $\beta>0$ such that $K<\frac{1}{2d+4\beta}$. We also fix a positive integer $p$ such that $dK+p(K(4\beta+2d)-1)<0$.\\
As before, we fix
$$\eps=\eps_{\hbar}:=|\log\hbar|^{-K}.$$
Combining the Bienaym\'e-Tchebychev inequality to Proposition~\ref{p:QE} and to inequality~\eqref{e:comparison}, the subset $J_{k,K,m}(\hbar)\subset\{1,\ldots, N(\hbar)\}$ defined as 
\begin{equation}\label{e:large-deviation}
J_{k,K,m}(\hbar):=\left\{j\in\{1,\ldots, N(\hbar)\}:\ \left|\frac{\int_{M}\chi_{x_k,m\eps}|\psi_{\hbar}^j|^2dv_g}{\int_{M}\chi_{x_k,m\eps}dv_g}-1\right|\geq\eps^{\beta}\right\},
\end{equation}
satisfies
\begin{equation}\label{e:density1}
\frac{\sharp J_{k,K,m}(\hbar)}{N(\hbar)}\leq \frac{C_0}{\eps^{(4\beta+2d)p}|\log\hbar|^p},
\end{equation}
where the positive constant $C_0$ depends only on $p$, on $\chi$, and on the manifold $(M,g)$. In particular, $C_0$ is uniform for $1\leq k\leq R(\eps)$. We then define $$\Lambda_{k,K,m}(\hbar):=\{1\leq j\leq N(\hbar):j\notin J_{k,K,m}(\hbar)\},$$ and
$$\Lambda_{K,m}(\hbar):=\bigcap_{k=1}^{R(\eps)}\Lambda_{k,K,m}(\hbar).$$
Thanks to~\eqref{e:number-eps-ball} and to~\eqref{e:density1}, we get
$$\frac{\sharp\Lambda_{K,m}(\hbar)}{N(\hbar)}\geq 1-\frac{1}{N(\hbar)}\sum_{k=1}^{R(\eps)}\sharp J_{k,K,m}(\hbar)\geq 1-\frac{C_0c_2}{\eps^{(4\beta+2d)p+d}|\log\hbar|^p}.$$
Hence using the definition of $\eps=\eps_{\hbar}$, we find that $\frac{\sharp\Lambda_{K,m}(\hbar)}{N(\hbar)}$ tends to $1$ as $\hbar$ goes to $0$. This in particular, using ~\eqref{e:large-deviation} and~\eqref{e:comparison}, shows that $\hbar>0$ small enough, for every $1\leq k\leq R(\eps)$, and for every $j\in \Lambda_{K,m}(\hbar)$,
$$a_1\eps^d\leq\int_{M}\chi_{x_k,m\eps}|\psi_{\hbar}^j|^2dv_g\leq a_2\eps^d,$$
where the constants $a_1,a_2>0$ depend only on $\chi$ and on the manifold $(M,g)$. By taking intersection $\Lambda_{K}(\hbar):=\Lambda_{K,1/2}(\hbar)\cap\Lambda_{K,50}(\hbar)$, we still have a full density subsequence. Hence we have proved the following crucial lemma
\begin{lemm}\label{l:density1} Suppose that $M$ is negatively curved. Let $0<K<\frac{1}{2d}$ and $\eps=|\log\hbar|^{-K}$. There exists $0<\hbar_0\leq 1/2$ such that, for every $0<\hbar\leq\hbar_0$, the following holds:\\
Given an orthonormal basis $(\psi_{\hbar}^j)_{j=1,\ldots N(\hbar)}$ of $\mathbf{1}_{[1-\alpha\hbar,1+\alpha\hbar]}(-\hbar^2\Delta_g)L^2(M)$ made of eigenfunctions of $-\hbar^2\Delta_g$, one can find a full density subset $\Lambda_K(\hbar)$ of $\{1,\ldots , N(\hbar)\}$ such that, for every $1\leq k\leq R(\eps)$, and for every $j\in \Lambda_K(\hbar)$, one has
\begin{equation}\label{e:vol-eps-ball}a_1\eps^d\leq\int_{B(x_k,\eps)}|\psi_{\hbar}^j|^2dv_g\leq\int_{B(x_k,50\eps)}|\psi_{\hbar}^j|^2dv_g\leq a_2\eps^d,
\end{equation}
where the constants $a_1,a_2>0$ depend only on $\chi$ and on the manifold $(M,g)$.
\end{lemm}
This lemma is the key ingredient that we will use to improve the usual Sogge $L^p$ upper bounds on eigenfunctions and the lower bounds on the size of nodal sets  along the full density subset.

\subsection{$L^p$-estimates}
\label{ss:Lp}

In this part we make use of semiclassical $L^p$ estimates for quasimodes as in \cite{Zw12} (Chapter $10$) to obtain our new $L^p$ estimates. In fact, we will prove something slightly stronger than what we stated in the introduction and show that the $L^p$ norm of the eigenfunctions in the balls $B(x_k, \eps)$ can be controlled. In the proof we will make use of our quantum ergodicity property~\eqref{e:vol-eps-ball}. The key result of this section is the following:
\begin{prop}\label{p:Lp-small-ball} Let $p=\frac{2(d+1)}{d-1}$ and $0<K<1/(2d)$. Fix $$\eps=|\log\hbar|^{-K}.$$
Then, there exists $C_0>0$ (depending only on $(M,g)$, $K$ and $\chi$) such that, for every $1\leq k\leq R(\eps)$ and for every $j\in \Lambda_K(\hbar)$, one has
$$\frac{1}{\eps^d}\int_{B(x_k,2\eps)}|\psi_{\hbar}^j |^pdv_g \leq C_0\frac{\eps}{\hbar}.$$
\end{prop}


Using~\eqref{e:number-eps-ball}, we easily deduce our main result on the $L^p$ norm of eigenfunctions of the Laplacian. Namely, there exists $C_K$ such that for all $j\in\Lambda_K(\hbar)$
$$\|\psi_{\hbar}^j \|_{L^p(M)}\leq C_K\left(\hbar^{-\frac{1}{p}}|\log\hbar|^{-\frac{K}{p}}\right),\quad p=\frac{2(d+1)}{d-1}.$$
We can then use an interpolation with $L^2$ norm to get our result for all $ 2<p< \frac{2(d+1)}{d-1} $.

\subsection{Proof of Proposition~\ref{p:Lp-small-ball}} In order to prove this Proposition, we will first draw a few consequences of our quantum ergodicity result. We will then show that, along $\Lambda_K(\hbar)$, and after rescaling using the change of coordinates $x= \text{exp}_{x_k}(\eps y)$, the restriction of the eigenfunctions to the rescaled balls are quasimodes of order 
\begin{equation}\label{h} h=\hbar/\eps, \end{equation} of a certain $h$-pseudodifferential operator (the rescale of $\Delta_g$). Finally we will use the semiclassical $L^p$ estimates from~\cite{Zw12} to establish an upper bound for the $L^p$ norm of these rescaled quasimodes.

\subsubsection{Preliminary observations}
 Let $1\leq k\leq R(\eps)$. Clearly
 $$\int_{B(x_k,2\eps)}|\psi_{\hbar}^j |^pdv_g \leq\int_{B(x_k,4\eps)}|\chi_{{x}_k,2\eps}\psi_{\hbar}^j |^pdv_g.$$
Using the change of coordinates $x= \text{exp}_{x_k}(\eps y)$, we get 
\begin{equation}\label{e:Lp-change-variable}\int_{B(x_k,2\eps)}|\psi_{\hbar}^j |^pdv_g \leq C_M\eps^{d}\int_{B_k(0,4)}|\chi(\|y\|_{x_k}/2)\psi_{\hbar}^j\circ\exp_{x_k}(\eps y) |^pdy,\end{equation}
where $B_k(0,4)$ is the ball centered at $0$ in $T_{x_k}M$ and $C_M$ depends only on the manifold $(M,g)$. We will now consider the following element of $L^2(B_k(0,4))$:
$$\tilde{u}_{\hbar,\eps}^j(y):=\chi(\|y\|_{x_k}/2)\psi_{\hbar}^j\circ\exp_{x_k}(\eps y),$$
and show that it is a quasimode of a certain order for a ``rescaling'' of the operator $-\hbar^2\Delta$.

Before that, we draw two simple consequences of the quantum ergodicity property. The first observation is that by \eqref{e:vol-eps-ball}, we get
\begin{equation}\label{e:Lp-QE1}\left\|\psi_{\hbar}^j\circ\exp_{x_k}(\eps y)\right\|_{L^2(B_k(0,50))}\leq C_Ma_2\leq c_M\frac{a_2}{a_1}\left\|\tilde{u}_{\eps,\hbar}^j\right\|_{L^2(\IR^d)},\end{equation}
where $C_M$ and $c_M$ depends only on $(M,g)$. The second observation is that clearly the first inequality also implies that
\begin{equation}\label{e:Lp-QE2}\left\|\tilde{u}_{\eps,\hbar}^j\right\|_{L^2(\IR^d)}=\ml{O}(1).\end{equation}

\subsubsection{Rescaled Laplacian}
We now introduce the following differential operator on $T_{x_k}M\simeq \IR^d$:
$$Q_h:=-h^2\chi(\|y\|_{x_k}/{20})\left(\sum_{i,j}g^{i,j}(\eps y)\frac{\partial^2}{\partial y_i\partial y_j}+\frac{1}{D_g(\eps y)}\frac{\partial}{\partial y_i}\left((D_g g^{i,j})(\eps y) \right)\frac{\partial }{\partial y_j}\right),$$
where $(g^{i,j})$ are the matrix coefficients of the metric $g$ in the coordinates $y=\kappa_k(x):=\exp_{x_k}^{-1}(x)$, and $D_g=\sqrt{\det(g_{i,j})}.$ We recall that $h=\frac{\hbar}{\eps}$.

Since $\psi_{\hbar}^j$ is an eigenmode of $-\hbar^2\Delta_g$ with eigenvalue $E$ in $[1-\alpha\hbar,1+\alpha\hbar]$, clearly one has
$$\forall y\in B_k(0,4),\ Q_h\left(\psi_{\hbar}^j\circ\kappa_k^{-1}(\eps y)\right)=E\psi_{\hbar}^j\circ\kappa_k^{-1}(\eps y).$$
This implies that 
$$\left\|(Q_h-E)\tilde{u}_{\eps,\hbar}^j\right\|_{L^2(\IR^d)}=\left\|[Q_h,\chi(\|.\|_{x_k}/2)]\psi_{\hbar,\eps}^j\right\|_{L^2(\IR^d)},$$
where $\psi_{\hbar,\eps}^j(y):=\chi(\|y\|_{x_k}/4)\psi_{\hbar}^j\circ\kappa_k^{-1}(\eps y)$. Using the composition laws for pseudodifferential operators in $\IR^d$, we get
$$[Q_h,\chi(\|.\|_{x_k}/2)]=h\Op_h^w(r_0)+h^2\Op_h^w(r_1),$$
$\Op_h^w$ denotes the Weyl quantization on $\IR^d$, $r_1\in S^{0,0}(\IR^{2d})$ and $r_0:=\frac{1}{i}\{q_0,\chi(\|.\|_{x_k}/2)\}$ in $S^{1,0}(\IR^{2d})$ with $q_0(y,\eta)=\chi(\|y\|_{x_k}/20)\sum_{i,j} g^{i,j}(\eps y)\eta_i\eta_j.$ Thanks to the Calder\'on-Vaillancourt theorem, we get
$$\left\|(Q_h-E)\tilde{u}_{\eps,\hbar}^j\right\|_{L^2(\IR^d)}\leq h\left\|\Op_h^w(r_0)\psi_{\hbar,\eps}^j\right\|_{L^2(\IR^d)}+Ch^2\left\|\psi_{\hbar,\eps}^j\right\|_{L^2(B_k(0,50))},$$
for some uniform constant $C>0$. Therefore by \eqref{e:Lp-QE1}, we get
\begin{equation}\label{e:QM-step0}\left\|(Q_h-E)\tilde{u}_{\eps,\hbar}^j\right\|_{L^2(\IR^d)}\leq h\left\|\Op_h^w(r_0)\psi_{\hbar,\eps}^j\right\|_{L^2(\IR^d)}+\ml{O}(h^2)\left\|\tilde{u}_{\eps,\hbar}^j\right\|_{L^2(\IR^d)}.\end{equation}

\subsubsection{Order of the quasimode}\label{sss:QM}

In order to show that $\tilde{u}_{\eps,\hbar}^j$ is in fact a $\ml{O}(h)$ quasimode for the operator $Q_h$, by \eqref{e:QM-step0} it just remains to show that $\left\|\Op_h^w(r_0)\psi_{\hbar,\eps}^j\right\|_{L^2(\IR^d)}$ is uniformly bounded. The only difficulty is that $\Op_h^w(r_0)$ is not a priori bounded on $L^2$ as $r_0$ belongs to $S^{1,0}(\IR^{2d})$. However we can easily overcome this issue by inserting an appropriate smooth cutoff function in $\xi$ variable as follows.

First, we write
\begin{equation}\label{e:QM-step1}\Op_h^w(r_0)\psi_{\hbar,\eps}^j(y)=T_{\eps}\Oph^w(\tilde{T}_{\eps}(r_0))\left((\kappa_k^{-1})^*(\chi_{x_k,4\eps}\psi_{\hbar}^j)\right),\end{equation}
where $\tilde{T}_{\eps}(r_0)(y,\eta):=r_0(y/\eps,\eta)$, and $T_{\eps}(f)(y):=f(\eps y)$. Hence
\begin{equation}\label{e:QM-step2}\left\|\Op_h^w(r_0)\psi_{\hbar,\eps}^j\right\|_{L^2(\IR^d)}=\eps^{-d/2}\left\|\Oph^w(\tilde{T}_{\eps}(r_0))\left((\kappa_k^{-1})^*(\chi_{x_k,4\eps}\psi_{\hbar}^j)\right)\right\|_{L^2(\IR^d)}.\end{equation}
Using Remark~\ref{r:cutoff}, we observe that $\chi_1(-\hbar^2\Delta_g)\psi_{\hbar}^j=\psi_{\hbar}^j$, and that $\chi_1(-\hbar^2\Delta_g)$ is an element of $\Psi^{-\infty}(M)$ with principal symbol $\theta(x,\xi)=\chi_1((\|\xi\|_x^*)^2)$ (\cite{Zw12}, Chapter $14$). Therefore, by replacing $\psi_{\hbar}^j$ with $\chi_1(-\hbar^2\Delta_g)\psi_{\hbar}^j$ and noting that $\tilde{T}_{\eps}(r_0)(\kappa_k^{-1})^*(\chi_{x_k, \eps}\theta)$ belongs to $\mathcal S (\R^{2d})$, we obtain
$$\left\|\Oph^w(\tilde{T}_{\eps}(r_0))\left((\kappa_k^{-1})^*(\chi_{x_k,4\eps}\psi_{\hbar}^j)\right)\right\|_{L^2(\IR^d)}\leq\ml{O}(1)\left\|\chi_{x_k,10\eps}\psi_{\hbar}^j\right\|_{L^2(M)} +\ml{O}(\hbar^{1-2\nu}).$$
where $0<\nu<1/2$ is such that $\hbar^{\nu}\leq\eps=\eps_{\hbar}$ for $\hbar$ small enough and where the constants are uniform for $1 \leq k\leq R(\eps)$.
Applying this upper bound to~\eqref{e:QM-step2}, and using Lemma~\ref{l:density1}, and the fact that $\eps=|\log\hbar|^{-K}$, we get
\begin{equation}\label{e:QM-step3}\left\|\Op_h^w(r_0)\psi_{\hbar,\eps}^j\right\|_{L^2(\IR^d)}\leq \ml{O}(1).\end{equation}
\begin{rema} We note that, in order to get a $\ml{O}(1)$, we need that $\hbar^{1-2\nu}\eps^{-\frac{d}{2}}$ remains bounded. 
As we will take $\eps=|\log\hbar|^{-K}$, this restriction does not matter. 
 
\end{rema}

Finally, combining this with \eqref{e:QM-step0} and \eqref{e:Lp-QE1}, we arrive at
$$\left\|(Q_h-E)\tilde{u}_{\eps,\hbar}^j\right\|_{L^2(\IR^d)}=\ml{O}(h)\left\|\tilde{u}_{\eps,\hbar}^j\right\|_{L^2(\IR^d)},$$
where the constant in the remainder is uniform  for $j$ in $\Lambda_K(\hbar)$ and for $1\leq k\leq R(\eps)$ (note that the rescaled quasimode $\tilde{u}_{\eps,\hbar}^j$ depends on $k$). This is precisely the definition of a quasimode of order $h$.

\subsubsection{Semiclassical $L^p$-estimates}

We are now in position to apply the semiclassical $L^p$ estimates from~\cite{Zw12} (precisely Theorem $10.10$), which are valid for $\ml{O}(h)$-quasimodes.

Using Remark~\ref{r:cutoff} and similar arguments as we delivered in paragraph~\ref{sss:QM}, we can see that the sequence $(\tilde{u}_{\eps,\hbar}^j)_{0<h\leq h_0}$ is localized in a compact subset $\ml{K}$ of phase space in the sense of section $8.4$ of \cite{Zw12}. For instance, we can take $\ml{K}$ to be 
$$\overline{B_k(0,5)}\times \{ \half \leq \xi^2 \leq 3/2 \}.$$
We then note that the principal symbol $$q_0-E=\chi(\|y\|^2_{x_k}/20)\sum g^{jk}(\eps y) \eta_j \eta_k -E,$$ of $Q_h -E$ satisfies:
\begin{itemize}
 \item $ \partial_{\eta} (q_0-E) \neq 0 \; \text{on} \; \{q_0=E\} \cap \ml{K}, $
 \item for all $x_0 \in \overline{B_k(0,5)}$, the hypersurface $\{q_0(x_0, \xi)=E\}$ has a nonzero second fundamental form (nonzero curvature).
\end{itemize}
Hence  all the required conditions of Theorem $10.10$ in \cite{Zw12}, for the semiclassical $L^p$ estimates to hold, are satisfied and we get
$$\| \tilde{u}_{\eps,\hbar}^j \|_{L^p(\IR^d)} \leq C h^{-1/p} \| \tilde{u}_{\eps,\hbar}^j \|_{L^2(\mathbb R^d)}, \qquad p=\frac{2(d+1)}{d-1}, $$
where $C$ only depends on $(M,g)$, and on a finite number of derivatives of  $\chi$, $\chi_1$, and the exponential map. Therefore, by \eqref{e:Lp-change-variable} and \eqref{e:Lp-QE2}, we get
$$ \eps^{-\frac{d}{p}}\| \psi_{\hbar}^j\|_{L^p(B(x_k, 2 \eps))} \leq C' \left(\frac{\eps}{\hbar}\right)^{\frac{1}{p}} , $$
where the constant is uniform for $j\in \Lambda_K(\hbar)$ and $1\leq k\leq R(\eps)$. We recall again that $h=\frac{\hbar}{\eps}$

\subsubsection{$L^{\infty}$ bounds}
\label{sss:Linfinity}

Using the same rescaling argument as in the last section and also the semiclassical $L^\infty$ estimates of \cite{KTZ07} and \cite{SmZw13}, we get
$$ \| \psi_{\hbar}^j\|_{L^\infty(B(x_k, 2 \eps))} \leq C \eps^{\frac{d-1}{2}} \hbar^{-\frac{d-1}{2}}, $$ where $C$ is again independent of $k$ and $j\in \Lambda_K(\hbar)$. We then have
$$\| \psi_{\hbar}^j\|_{L^\infty(M)} =\max_{k} \{\| \psi_{\hbar}^j\|_{L^\infty(B(x_k, 2 \eps))} \} \leq C\eps^{\frac{d-1}{2}} \hbar^{-\frac{d-1}{2}}.$$  This proves our claimed $L^\infty$ bounds.

\subsection{Lower bounds for nodal sets on a fixed open set $U$}

Recall that we want to give a lower bound on the size of nodal sets on every fixed open set $U$. If we were only interested in the case $U=M$, then we could conclude more quickly using the results of~\cite{SoZe11, HS12} (see paragraph \ref{sss:HezSog}).

\subsubsection{Colding-Minicozzi approach}

The main lines of our argument are similar to the ones in~\cite{CM}. The proof in this reference was based on the study of the so-called $q$-good balls. Namely, given a constant $q>1$ and an eigenmode $\psi_{\hbar}$ of the operator $-\hbar^2 \Delta_g$, a ball $B(x,r)\subset M$ (centered at $x$ and of radius $r>0$) is said to be $q$-good if
$$\int_{B(x,2r)}|\psi_{\hbar}|^2dv_g\leq 2^q\int_{B(x,r)}|\psi_{\hbar}|^2dv_g.$$
A ball which is not $q$-good is said to be $q$-bad. The main advantage of $q$-good balls is that you can control from below the volume of the nodal set on them. Precisely, it was shown in~\cite{CM} (Proposition $1$) that
\begin{prop}\label{p:CM} Let $q>1$ and let $r_0>1$. There exist $\mu>0$ and $\hbar_0>0$ so that if $0<\hbar\leq\hbar_0$, $E\in[1-\alpha\hbar,1+\alpha\hbar]$, $-\hbar^2\Delta_g\psi_{\hbar}=E\psi_{\hbar}$ on $B(x,r)\subset M$ with $r\leq r_0\hbar$, $\psi_{\hbar}$ vanishes somewhere on $B(x,r/3)$, and 
$$\int_{B(x,2r)}|\psi_{\hbar}|^2dv_g\leq 2^q\int_{B(x,r)}|\psi_{\hbar}|^2dv_g,$$
then
$$\ml{H}^{n-1}\left(B(x,r)\cap\{\psi_{\hbar}=0\}\right)\geq\mu r^{n-1}.$$
\end{prop}

\begin{rema}
 We stress that this proposition is valid for any smooth compact Riemannian manifold without any assumption on the curvature.
\end{rema}

This proposition provides a local estimate on the nodal set provided that we are on a $q$-good ball on which $\psi_{\hbar}$ vanishes. We now state a result of  Courant which guarantees such a vanishing property (for a proof see for example Lemma $1$ of \cite{CM})
\begin{lemm}\label{l:CM}
 There exists $r_0>0$ so that if $0<\hbar\leq 1/2$, $E\in[1-\alpha\hbar,1+\alpha\hbar]$, $-\hbar^2 \Delta_g\psi_{\hbar}=E\psi_{\hbar}$, then $\psi_{\hbar}$ has a zero in every ball of radius $\frac{r_0\hbar}{3}$.
\end{lemm}

By applying these two properties to a proper cover of $M$ with balls of radius $r_0\hbar$, the proof of \cite{CM} boils down to counting the number of $q$-good balls in a cover (for some large enough $q>0$).

Our plan is to use a similar counting argument as \cite{CM} except that we will take advantage of the small-scale quantum ergodic property of eigenfunctions. For such eigenfunctions, we will be able to count the number of $q$-good balls in balls of radius $\eps$ of order $|\log\hbar|^{-K}$ (with $0<K<1/(2d)$).

\subsubsection{Lower bounds on the number of $q$-good balls}

First, we cover $M$ with balls $(B(\tilde{x}_k, r_0\hbar))_{k=1,\ldots, \tilde{R}(\hbar)}$ with $r_0$ defined in Lemma \ref{l:CM}. In particular, any eigenfunction $\psi_{\hbar}$ of $-\hbar^2 \Delta_g$ on the space $\mathbf{1}_{[1-\alpha\hbar,1+\alpha\hbar]}(-\hbar^2 \Delta_g)L^2(M)$ has a zero in every ball $B(\tilde{x}_k,r_0\hbar/3)$. In addition, we require that the covering is chosen in such a way that each point in $M$ is contained in $C_g$-many of the double balls $B(\tilde{x}_k, 2r_0\hbar)$. Recall that the number $C_g$ can be chosen to be only dependent on $(M,g)$, and hence independent of $0<\hbar\leq 1/2$~\cite{CM}.

We now fix $1\leq k\leq R(\eps)$ and estimate the number of $q$-good balls \footnote{We underline that we count the number of $q$-good balls for a given eigenfunction in the full density subsequence obtained in Lemma \ref{l:density1}. Also in~\cite{CM}, the authors count the number of $q$-good balls over $M$.} inside $B(x_k, 2\epsilon)$.

Let $\{B(\tilde{x}_l, r_0\hbar):l\in A_k\}$ be the set of those balls in the covering which have non-empty intersection with $B(x_k, \eps)$. Then, taking $\hbar>0$ small enough to ensure $4r_0\hbar\leq\eps$, one has
$$B(x_k, \eps) \subset \bigcup_{l\in A_k} B(\tilde{x}_l, r_0\hbar) \subset \bigcup_{l\in A_k} B(\tilde{x}_l, 2r_0\hbar) \subset B(x_k, 2\eps).$$ 
We denote by $A_{q-\text{good}}^k$ the subset of indices such that $l\in A_k$ and $B(\tilde{x}_l, r_0\hbar)$ is a $q$-good ball. Then, we let $\mathcal G_k= \bigcup_{l\in A_{q-\text{good}}^k} B(\tilde{x}_l, r_0\hbar).$ We start our proof with the following lemma which is an analogue of Lemma $3$ of \cite{CM} for balls of radius $\eps$:
\begin{lemm}\label{l:average-good} There exists $q_M> 1$ depending only on $(M,g)$ such that, for every $\hbar>0$ small enough, and, for every $j\in \Lambda_K(\hbar)$, one has
\begin{equation} \label{MassOfG} \quad \int_{\mathcal G_k}|\psi_{\hbar}^j |^2dv_g \geq \half a_1 \eps^{d},\end{equation} 
where $a_1>0$ is the same constant as in Lemma~\ref{l:density1}.
\end{lemm}


\begin{proof}
 To show this we first note that, using Lemma~\ref{l:density1},
$$\int_{\mathcal G_k}|\psi_{\hbar}^j |^2dv_g \geq \int_{B(x_k, \epsilon)}|\psi_{\hbar}^j |^2dv_g -\int_{\mathcal B_k}|\psi_{\hbar}^j |^2dv_g \geq a_1 \eps^{d} -\int_{\mathcal B_k}|\psi_{\hbar}^j |^2dv_g,$$ where $\mathcal B_k=\bigcup_{l\in A_k-A_{q-\text{good}}^k} B(\tilde{x}_l,r_0\hbar)$. On the other hand, using Lemma~\ref{l:density1} one more time, one has
\begin{align*} \int_{\mathcal B_k}|\psi_{\hbar}^j |^2 dv_g & \leq  \sum_{l\in A_k-A_{q-\text{good}}^k} \int_{B(\tilde{x}_l, r_0\hbar)}|\psi_{\hbar}^j |^2dv_g \\ 
&   <  2^{-q} \sum_{l\in A_k-A_{q-\text{good}}^k} \int_{B(\tilde{x}_l, 2r_0\hbar)}|\psi_{\hbar}^j |^2 dv_g\\
& \leq  C_g2^{-q} \int_{B(x_k, 2 \eps)}|\psi_{\hbar}^j  |^2dv_g\\
& \leq  C_g a_2 2^{-q}\eps^{d}.
\end{align*}
Therefore, by choosing $q$ large enough so that $C_g a_2 2^{-q} \leq \half a_1$, we have proved~\eqref{MassOfG}.
\end{proof}

To estimate the number of $q$-good balls which intersect $B(x_k,\eps)$, we will estimate $\text{vol}(\mathcal G_k)$. Unfortunately, since $\mathcal G_k$ is a complicated set and, since it depends on $\hbar$ in a way which is hardly tractable, we cannot use quantum ergodicity to get a lower bound for $\text{vol}(\mathcal G_k)$. Instead, we will use our $L^p$ estimates \footnote{ Sogge's $L^p$ estimates were already used in the proof of \cite{CM}, but here we use our slightly stronger version.} from Proposition \ref{p:Lp-small-ball}. First, we write H\"older's inequality with $p=\frac{2(d+1)}{(d-1)}$ to get
\begin{equation} \label{VolG}\half a_1 \eps^{d} \leq \int_{\mathcal G_k}|\psi_{\hbar}^j |^2dv_g \leq \big( \text{vol}(\mathcal G_k) \big)^{\frac{p-2}{p}} || \psi_{\hbar}^j ||^2_{L^p(B(x_k, 2 \eps))}. \end{equation} 
By applying Proposition~\ref{p:Lp-small-ball} to \eqref{VolG}, we find the following lower bound on the volume of $\mathcal{G}_k$:
$$ \text{vol}(\mathcal G_k)\geq a_0\eps^{\frac{d+1}{2}}\hbar^{\frac{d-1}{2}},$$
for some positive constant $a_0$ that depends only on $(M,g)$ and $\chi$. This implies that
\begin{equation}\label{e:cardinal-good}\sharp A^k_{q-\text{good}}\geq \tilde{a}_0\eps^{\frac{d+1}{2}}\hbar^{-\frac{d+1}{2}},\end{equation}
where $\tilde{a}_0>0$ and depends only on $(M,g)$ and $\chi$.

\subsubsection{Conclusion} By combining Proposition \ref{p:CM} with the lower bound~\eqref{e:cardinal-good}, we find that 
for every $0<\hbar\leq\hbar_0$, for every $1\leq k\leq \tilde{R}(\eps)$, and for every $j\in\Lambda_K(\hbar)$, we have
\begin{equation} \label{e:mainineq} c_0\left(\frac{\eps}{\hbar}\right)^{\frac{d+1}{2}}\mu\hbar^{d-1}\leq\sum_{l\in A^k_{q-\text{good}}}\mathcal{H} ^{d-1} \left( Z_{\psi_{\hbar}^j} \cap B(\tilde{x}_l, r_0\hbar)\right) \leq C_g\mathcal{H} ^{d-1} \left( Z_{\psi_{\hbar}^j} \cap B(x_k, 2\eps)\right), \end{equation} 

To finish the proof of Theorem 1, let $U_r$ be an arbitrary open ball of radius $r$ (independent of $\hbar$), and $U_{r/2}$ be a ball of radius $r/2$ concentric to $U_r$. Also let $\{B(x_k, \eps):k \in A\} $ be the balls with nonempty intersection with $U_{\frac{r}{2}}$. Clearly 
$$ U_{\frac{r}{2}} \subset \bigcup_{k \in A} B(x_k, 2\eps) \subset U_r, $$
and $A$ must have at least $c \eps^{-d}$ elements (for some uniform constant $c>0$ depending on $r$). Hence using (\ref{e:mainineq}), 
$$  \mathcal{H} ^{d-1} ( Z_{\psi_{E_j(\hbar)}} \cap U_r ) \geq C \hbar^{\frac{d-3}{2}} |\log \hbar|^{\frac{K(d-1)}{2}}.$$ 
Since $0<K<\frac{1}{2d}$, we get the claimed range of exponents for our logarithmic factor.

\subsubsection{Alternative proof for $U=M$}
\label{sss:HezSog}
When $U=M$, we can conclude more quickly using inequality~\eqref{e:HeSo} from the introduction. In fact, thanks to the H\"older's inequality, one has
$$1=\left\|\psi_{\hbar}^j\right\|_{L^2}^{\frac{1}{\theta}}\leq\|\psi_{\hbar}^j\|_{L^1(M)}\|\psi_{\hbar}^j\|_{L^p(M)}^{\frac{1}{\theta}-1},$$
where $p=\frac{2(d+1)}{d-1}$ and $\theta=\frac{p-2}{2(p-1)}$. This inequality was already used in~\cite{SoZe11} (end of paragraph $1.1$). Hence by applying our improved $L^p$ estimates, we find that
$$C_1\hbar^{\frac{d-1}{4}}|\log\hbar|^{\frac{K(d-1)}{4}}\leq \|\psi_{\hbar}^j\|_{L^1(M)} ,$$
for some uniform constant $C_1>0$. The conclusion now follows from inequality \eqref{e:HeSo} and the fact that we can allow any exponent $0<K<\frac{1}{2d}$.


\appendix

\section{Moments of order $2p$}
\label{a:moments}

In this appendix, we give the proof of Proposition \ref{p:moments}. Before that, we need to recall a few results and notations from \cite{Li04} -- see 
paragraphs \ref{sss:anosov} and \ref{sss:banach}. We deduce from these properties a multicorrelation result in paragraph~\ref{sss:multicorrelation} which
generalizes Theorem~\ref{t:liverani2}. Finally, using an argument similar to Lemma $3.2$ of \cite{Ra73}, we deduce Proposition~\ref{p:moments}.

\subsection{Anosov property}\label{sss:anosov}
When $M$ is a negatively curved manifold, the geodesic flow\footnote{In this case, $(G^t)_{t\in\IR}$ is the Hamiltonian flow associated to 
$p_0(x,\xi):=\frac{(|\xi\|_x^*)^2}{2}.$} $(G^t)_{t\in\IR}$ satisfies the Anosov property on $S^*M$. Precisely, it means that, for every $\rho=(x,\xi)$ in $S^*M$, 
there exists a $G^t$-invariant splitting
\begin{equation}\label{e:anosov}
T_{\rho}T_{1/2}^*M=\IR X_0(\rho)\oplus E^s(\rho)\oplus E^u(\rho),
\end{equation}
where $X_0(\rho)$ is the Hamiltonian vector field associated to $p_0(x,\xi)=\frac{(\|\xi\|_x^*)^2}{2}$, $E^u(\rho)$ is the unstable direction and $E^s(\rho)$ is 
the stable direction. These three subspaces are preserved under the geodesic flow and there exist constants $C_0>0$ and $\gamma_0>0$ such that, for any 
$t\geq 0$, for any $v^s\in E^s(\rho)$ and any $v^u\in E^u(\rho)$,
$$\|d_{\rho}G_0^tv^s\|_{G^t_0(\rho)}\leq C_0e^{-\gamma_0 t}\|v^s\|_{\rho}\ \text{and}\ \|d_{\rho}G_0^{-t}v^u\|_{G^{-t}_0(\rho)}\leq C_0e^{-\gamma_0 t}\|v^u\|_{\rho},$$
where $\|.\|_w$ is the norm associated to the Sasaki metric on $S^*M$. 

\subsection{Banach spaces adapted to the transfer operator}\label{sss:banach}
In this setting, one can study the spectral properties of the transfer operator which is defined as follows:
$$\forall f\in\ml{C}^{\infty}(S^*M),\ \ml{L}_t f=f\circ G^{-t}.$$
This spectral analysis has a long history and many progresses have been made recently \cite{Ra87, Dol98, Li04, Ts10, FaSj11, Ts12, BaLi12, GiLiPo13, DyZw13, FaTs13, NoZw13}. We will not give any details on these advancements and we refer the interested reader to the above references. 

In the present article, we will only make use of the results from~\cite{Li04}, and for that purpose, we recall a few definitions 
from this reference. We fix $\sigma\in(0,\gamma_0)$, and we introduce the following dynamical distances between two points $\rho,\rho'\in S^*M$:
$$\ d^s(\rho,\rho'):=\int_0^{+\infty}e^{\sigma t} d_{S^*M}(G^t(\rho),G^t(\rho'))dt,$$
and
$$d^u(\rho,\rho'):=\int_{-\infty}^0e^{-\sigma t} d_{S^*M}(G^t(\rho),G^t(\rho'))dt.$$
According to Lemma $2.3$ in~\cite{Li04}, $d^u$ is a pseudo-distance on $S^*M$ (meaning that it can take the value $+\infty$). Moreover, $d^u$, restricted to any 
strong-unstable manifold, is a smooth function and it is equivalent to the restriction of the Riemannian metric, while points belonging
to different unstable manifolds are at an infinite distance. The analogous properties hold for $d^s$. These distances allow to define different norms which are 
well adapted to the study of the transfer operator $\ml{L}_t$. First, for $0<\beta<1$, $\delta>0$ small enough and for every $f$ in $\ml{C}^1(S^*M)$, we define
$$H_{s,\beta}(f):=\sup_{d^s(\rho,\rho')\leq\delta}\frac{|f(\rho)-f(\rho')|}{d^s(\rho,\rho')^{\beta}},$$
and
$$|f|_{s,\beta}:=\|f\|_{\ml{C}^0}+H_{s,\beta}(f).$$
We observe that this norm is controlled by the standard $\beta$-H\"older norm, i.e. $|f|_{s,\beta}\leq\|f\|_{\ml{C}^{\beta}}.$ Then, we define $\ml{C}^{\beta}_s(S^*M)\subset 
\ml{C}^{0}(S^*M)$ as the completion of $\ml{C}^{1}(S^*M)$ with respect to the norm $|.|_{s,\beta}$. We also introduce a family of test functions
$$\ml{D}_{\beta}:=\{f\in\ml{C}^{\beta}_s(S^*M):|f|_{s,\beta}\leq 1\}.$$ Finally, we define the following norm, for every $f$ in $\ml{C}^{1}(S^*M)$:
$$\|f\|_{\beta}:=\|f\|_u+\|f\|_s,$$
with
$$\|f\|_s:=\sup_{g\in\ml{D}_{\beta}}\left|\int_{S^*M}fgdL\right|,\ \text{and}\ \|f\|_u:=H_{u,\beta}(f),$$
where $L$ is the disintegration of the Liouville measure on $S^*M$, and $H_{u,\beta}$ is defined similarly as $H_{s,\beta}$ except that we replace $d^s$ by $d^u$. We observe that
\begin{equation}\label{e:control-norms}
 \|f\|_s\leq \|f\|_{L^1}\leq\|f\|_{\ml{C}^{0}}.
\end{equation}
In order to prove his main result on the exponential decay of correlations, Liverani proved the following estimate (end of page~$1284$ in~\cite{Li04}):
\begin{equation}\label{e:spectral-gap}
\forall t\geq 0,\ \left\|\ml{L}_t(f)\right\|_s \leq C_{\beta} e^{-\sigma_{\beta} t}\left(\|X_0^2.f\|_{\beta}+\|X_0.f\|_{\beta}+\|f\|_{\beta}\right)
\end{equation}
 for every $\beta>0$ small enough, for every $f\in\ml{C}^3(S^*M)$ satisfying $\int_{S^*M}fdL=0$, and for some positive constants $C_{\beta}>0$ and $\sigma_{\beta}>0$ (independent of $f$ and $t$). 

\subsection{Multi-correlations}\label{sss:multicorrelation}

Thanks to these properties, we now prove the following preliminary lemma:
\begin{lemm}\label{l:multicorrelation} Let $p$ be a positive integer and let $0<\beta<1$. Then, there exists $\sigma_{p,\beta}>0$ 
and $C_{p,\beta}>0$ such that the following holds:\\
For every $\tau>0$, for every $f$ in $\ml{C}^{\beta}(S^*M)$ satisfying $\int_{S^*M}fdL=0$, and for every $0\leq t_1\leq t_2\leq\ldots\leq t_{2p}$ 
satisfying
$$\exists 1\leq j_0\leq 2p\ \text{such that}\ \forall k\neq j_0,\ |t_k-t_{j_0}|\geq\tau,$$ 
 one has
 $$\left|\int_{S^*M}\prod_{j=1}^{2p} f\circ G^{t_j}dL\right|\leq C_{p,\beta}e^{-\sigma_{p,\beta}\tau}\|f\|_{\ml{C}^{\beta}}^{2p}.$$
\end{lemm}

\begin{proof} We start with the case of $\ml{C}^3$ observables. Let $f$ be an element in $\ml{C}^3(S^*M)$ such that $\int_{S^*M}fdL=0$, and let $\tau>0$. We fix\footnote{The case $j_0=2p$ will be treated below.} $1\leq j_0\leq 2p-1$ such that $\forall k\neq j_0,\ |t_k-t_{j_0}|\geq\tau.$ We define 
$$f_{j_0}^+:=\prod_{j=1}^{j_0} f\circ G^{t_j-t_{j_0}},\ \text{and}, f_{j_0}^-:=\prod_{j=j_0+1}^{2p} f\circ G^{t_j-t_{j_0+1}}.$$
We set $\tilde{f}_{j_0}^+:=f_{j_0}^+-\int_{S^*M}f_{j_0}^+dL$. Then, we write
 $$\int_{S^*M}\prod_{j=1}^{2p} f\circ G^{t_j}dL=\int_{S^*M}f_{j_0}^+\circ G^{t_{j_0}-t_{j_0+1}}
f_{j_0}^-dL,$$
and thus
 $$\int_{S^*M}\prod_{j=1}^{2p} f\circ G^{t_j}dL=\int_{S^*M}f_{j_0}^+dL\int_{S^*M}f_{j_0}^-dL+ \int_{S^*M}\tilde{f}_{j_0}^+\circ G^{t_{j_0}-t_{j_0+1}}
f_{j_0}^-dL.$$
We now use the norms defined above to give an estimate on the remainder, i.e.
$$\left|\int_{S^*M}\tilde{f}_{j_0}^+\circ G^{t_{j_0}-t_{j_0+1}}
f_{j_0}^-dL\right|\leq\left\|\ml{L}_{t_{j_0+1}-t_{j_0}}\left(\tilde{f}_{j_0}^+\right)\right\|_s
H_{s,\beta}\left(f_{j_0}^-\right).$$
Since $t_j-t_{j_0+1}\geq 0$ for every $j\geq j_0+1$, we get from the definition of $H_{s,\beta}$ that 
$$H_{s,\beta}\left(f_{j_0}^-\right)=\ml{O}(1)\|f\|_{\ml{C}^{\beta}}^{2p-j_0},$$
where the constant does not depend on $f$ and $t=(t_1,\ldots,t_{2p})$. Similarly, as $t_j-t_{j_0}\leq 0$ for $j\leq j_0$, using the definition of $H_{u,\beta}$ and inequality~\eqref{e:control-norms}, one can verify that, for every $l=0,1,2$, one has $\|X_0^l.\tilde{f}_{j_0}^+\|_{\beta}=\ml{O}(1)\|f\|_{\ml{C}^{3}}^{j_0}$. Applying these observations to Liverani's estimate~\eqref{e:spectral-gap}, we get that
 $$\int_{S^*M}\prod_{j=1}^{2p} f\circ G^{t_j}dL=\int_{S^*M}f_{j_0}^+dL\int_{S^*M}f_{j_0}^-dL+\ml{O}(1)e^{-\sigma_0(t_{j_0+1}-t_{j_0})}\|f\|_{\ml{C}^{3}}^{2p},$$
where the constant $\ml{O}(1)$ does not depend on $f$ and $t=(t_1,\ldots,t_{2p})$. If $j_0=1$, this proves the result. In the case $2\leq j_0\leq 2p$, we have to analyze the term $\int_{S^*M}f_{j_0}^+dL$. For that purpose, we can write 
$$\int_{S^*M}\prod_{j=1}^{j_0} f\circ G^{t_j-t_{j_0}}dL=\int_{S^*M}\left(\prod_{j=1}^{j_0-1} f\circ G^{t_j-t_{j_0-1}}\right)\circ G^{t_{j_0-1}-t_{j_0}} fdL,$$
and perform the same analysis to get the expected upper bound. Thus, we have verified that there exists $C_p>0$ and $\sigma_p>0$ such that, for every $\tau>0$, for every $f$ in $\ml{C}^{3}(S^*M)$ satisfying $\int_{S^*M}fdL=0$, and for every $0\leq t_1\leq t_2\leq\ldots\leq t_{2p}$ 
satisfying
$$\exists 1\leq j_0\leq 2p\ \text{such that}\ \forall k\neq j_0,\ |t_k-t_{j_0}|\geq\tau,$$ 
one has
\begin{equation}\label{e:multiC3}\left|\int_{S^*M}\prod_{j=1}^{2p} f\circ G^{t_j}dL\right|\leq C_{p}e^{-\sigma_{p}\tau}\|f\|_{\ml{C}^{3}}^{2p}.
\end{equation}
However recall that we are interested in getting a control in terms of H\"older norms. To do this, we proceed as in Corollary $1$ of \cite{Dol98}. We fix $0<\beta<1$ and $a$ in $\ml{C}^{\beta}(S^*M)$ such that $\int_{S^*M} adL=0$. Fix $\gamma=\frac{\sigma_p}{10p}$. 
Using a convolution by a smooth function, one can obtain a function $f$ in $\ml{C}^3(S^*M)$ such that
$$\|a-f\|_{\ml{C}^0}\leq e^{-\beta \gamma \tau}\|a\|_{\ml{C}^{\beta}},\ \text{and}\ \|f\|_{\ml{C}^3}\leq e^{3\gamma \tau}\|a\|_{\ml{C}^{\beta}}.$$
If we let $\tilde{f}=f-\int_{S^*M}fdL$, then it follows that there exists some constant $C(p)>0$ such that
$$\left|\int_{S^*M}\prod_{j=1}^{2p} a\circ G^{t_j}dL-\int_{S^*M}\prod_{j=1}^{2p} \tilde{f}\circ G^{t_j}dL\right|\leq C(p) e^{-\beta\gamma\tau}\|a\|_{\ml{C}^{\beta}}^{2p}.$$
Applying~\eqref{e:multiC3}, we find that there exists some constant $C(p,\beta)>0$ such that
$$\left|\int_{S^*M}\prod_{j=1}^{2p} a\circ G^{t_j}dL\right|\leq C(p,\beta) (e^{(6p\gamma -\sigma_p)\tau}+e^{-\beta\gamma\tau})\|a\|_{\ml{C}^{\beta}}^{2p}.$$
Since $\gamma=\frac{\sigma_p}{10p}$, the Lemma follows.
\end{proof}

\subsection{Proof of Proposition~\ref{p:moments}} Recall that we want to prove that, for every $0<\beta<1$, there exist a constant 
$C_{\beta,p}>0$ such that, for every $T>0$, and for every $a\in\ml{C}^{\beta}(S^*M)$,
$$\tilde{V}_p(a,T):=\int_{S^*M}\left|\int_0^Ta\circ G^tdt-\int_{S^*M}adL\right|^{2p}dL\leq C_{\beta,p}\|a\|_{\ml{C}^{\beta}}^{2p}T^p.$$

Without loss of generality, we can suppose that $\int_{S^*M}a dL=0$ and that $a$ is real valued. We write
$$\tilde{V}_p(a,T)=\int_{[0,T]^{2p}}\left(\int_{S^*M}\prod_{j=1}^{2p} a\circ G^{t_j}dL\right) dt_1\ldots dt_{2p}.$$
We now follow the idea of Lemma $3.2$ of \cite{Ra73}, and define
$$A_1(T):=\left\{(t_1,\ldots, t_{2p})\in[0,T]^{2p}:\ \forall 1\leq j\leq 2p,\ \exists k\neq j: |t_j-t_k|< 1\right\}.$$
For $n\geq 2$, we also define
$$B_n(T):=\left\{(t_1,\ldots, t_{2p})\in [0, T]^{2p} :\ \forall 1\leq j\leq 2p,\ \exists k\neq j: |t_j-t_k|< n\right\},$$and
$$ A_n(T)=B_n(T)-B_{n-1}(T). $$
We then clearly have
$$\tilde{V}_p(a,T)\leq\sum_{n\geq 1}\int_{A_n(T)}\left|\int_{S^*M}\prod_{j=1}^{2p} a\circ G^{t_j}dL\right| dt_1\ldots dt_{2p},$$
which by Lemma~\ref{l:multicorrelation} gives 
\begin{equation}\label{e:concl-variance}\tilde{V}_p(a,T)\leq C_{p,\beta}\left(\text{Leb}(A_1(T)) +\sum_{n\geq 2}\text{Leb}(A_n(T))e^{-\sigma_{p,\beta}{(n-1)}}\right)\|a\|_{\ml{C}^{\beta}}^{2p}.\end{equation}
It now remains to estimate $\text{Leb}(A_n(T))$ for every $n\geq 1$. For every $n\geq 1$, using a simple rescaling we get
\begin{equation}\label{A_n} \text{Leb}(A_n(T)) \leq \text{Leb}(B_n(T)) = n^{2p} \text{Leb}(A_1(T/n)) \leq n^{2p} \text{Leb}(A_1(T)). \end{equation}
Hence we only have to estimate $\text{Leb}(A_1(T))$. To do this we write
$$\text{Leb}(A_1(T))\leq \int_{[0,T]^{2d}}\prod_{j=1}^{2p}\left(\sum_{k\neq j}\mathbf{1}_{[-1,1]}(t_j-t_k)\right)dt_1\ldots d_{t_{2p}}.$$ 
The function $\prod_{j=1}^{2p}\left(\sum_{k\neq j}\mathbf{1}_{[-1,1]}(t_j-t_k)\right)$ does not vanish if and only if we can split the $2p$-uple 
$(t_1,\ldots, t_{2p})$ into $N$ disjoint subfamilies $(t_{\alpha_1^1},\ldots,t_{\alpha_{n(1)}^1}),\ldots ,(t_{\alpha_1^N},\ldots,t_{\alpha_{n(N)}^N})$ such that
\begin{itemize}
\item for every $1\leq j\leq N$, $n(j)\geq 2$, 
\item for every $t_{\alpha_r^j}$ element in the $j$-th subfamily, there exists $1\leq r'\neq r\leq n(j)$ such that $|t_{\alpha_{r}^j}-t_{\alpha_{r'}^j}|\leq 1$.
\end{itemize}
It implies that there exists some universal constant $c_p>0$ (depending only on $p$) such that $\text{Leb}(A_1(T))\leq c_pT^p.$ This, together with \eqref{e:concl-variance} and \eqref{A_n}, implies that
$$\tilde{V}_p(a,T)\leq \tilde{C}_{p}T^p\|a\|_{\ml{C}^{\beta}}^{2p}.$$

\section*{Acknowledgments}

The first author is partially supported by the NSF grant DMS-0969745. The second author is partially supported by the Agence Nationale de la Recherche through the Labex CEMPI (ANR-11-LABX-0007-01) and the ANR project GeRaSic (ANR-13-BS01-0007-01). 


\end{document}